\documentclass[a4wide,12pt]{article}

\usepackage{lipsum}
\usepackage{amsfonts}
\usepackage{graphicx}
\usepackage{epstopdf}
\usepackage{stmaryrd}
\usepackage{cite}
\ifpdf
  \DeclareGraphicsExtensions{.eps,.pdf,.png,.jpg}
\else
  \DeclareGraphicsExtensions{.eps}
\fi

\usepackage{calc}
\usepackage{microtype}
\usepackage{amsmath}
\usepackage{amssymb}
\usepackage{amsthm}
\usepackage{color}

\usepackage{caption}
\usepackage{subcaption}
\usepackage{graphicx}
\usepackage[hidelinks]{hyperref}
\hypersetup{
    colorlinks=true,
    linkcolor=blue,
    filecolor=blue,      
    urlcolor=blue,
    citecolor = blue,
    }
\usepackage{tikz}
\usetikzlibrary{plotmarks}
\usepackage{pgfplots}
\usepackage{dsfont}
\usepackage{bm}
\usepackage{listings}
\usepackage{cleveref}
\usepackage{comment}
\usepackage{cite}
\usepackage{mathrsfs,graphicx,color,float, indentfirst,textcomp}
\usepackage{setspace}
\usepackage{latexsym,lscape,amsfonts}
\usepackage{enumerate}
\usepackage[shortlabels]{enumitem}
\usepackage{lscape}
\usepackage{multirow} 
\usepackage{xcolor}
\usepackage{geometry}
\usepackage{mathtools}
\usepackage[normalem]{ulem}
\usepackage[algoruled,boxed,lined]{algorithm2e}

\theoremstyle{definition}
\numberwithin{equation}{section}
\newtheorem{theorem}{Theorem}[section]
\newtheorem{remark}[theorem]{Remark}

\newtheorem{lemma}{Lemma}[section]

\newtheorem{proposition}[theorem]{Proposition}

\title{A fast spectral particle method for the Landau equation}

\author{Giacomo Borghi\thanks{Corresponding author.}$\,\,$\thanks{ Department of Mathematics, Technical University of Munich \& Munich Center for Machine Learning, Munich, Germany (\texttt{giacomo.borghi@tum.de})}
\and Lorenzo Pareschi\thanks{Department of Mathematics and Computer Science, University of Ferrara, Ferrara, Italy.}
\thanks{Maxwell Institute for Mathematical Sciences and Department of Mathematics, School of Mathematical and Computer Sciences (MACS), Heriot-Watt University, Edinburgh, UK
  (\texttt{l.pareschi@hw.ac.uk}).}
}

\usepackage{amsopn}

\usepackage{mathrsfs}
\usepackage{mathtools}
\usepackage{amssymb}
\usepackage[shortlabels]{enumitem}
\usepackage{bm}

\newcommand{\Rd}{ {\mathbb{R}^d}}

\renewcommand{\P}{\mathcal{P}}

\newcommand{\Td}{\mathbb{T}^d}

\newcommand{\ve}{\varepsilon}

\newcommand{\supp}{\mathrm{supp}}

\newcommand{\per}{{\mathrm{per}}}

\newcommand{\be}{\begin{equation}}
\newcommand{\ee}{\end{equation}}

\ifpdf
\hypersetup{
  pdftitle={},
  pdfauthor={G. Borghi, L. Pareschi}
}
\fi

\begin{document}

\maketitle

\begin{abstract}
We propose a fast deterministic particle method for the spatially
homogeneous Landau equation. The method combines a particle
representation of the solution with a Fourier approximation of the
nonlinear collision flux. Nonuniform fast Fourier transforms are used
to reconstruct the density from the particles and to evaluate the flux
and density at the particle locations, while the convolutional
structure of the flux enables its efficient computation by FFTs.
For a fixed transform tolerance, the cost per time step is
$\mathcal{O}(N+M^d\log M)$, where $N$ is the number of particles and
$M$ the number of Fourier modes per velocity dimension.
We establish a consistency estimate for
the reconstructed velocity field on regions where the reference
density is bounded away from zero. The estimate separates the spectral
truncation error from the particle-density reconstruction error, showing how the latter can dominate for sufficiently smooth densities.
Two-dimensional numerical experiments with Maxwellian interactions
illustrate the accuracy and efficiency of the method, its sensitivity
to particle and spectral resolution, and the effects of filtering.
Comparisons with a direct blob implementation demonstrate improved
accuracy at a strongly reduced computational cost.
\end{abstract}

\bigskip

  

\tableofcontents

\medskip


\section{Introduction}

The Landau equation arises as a limiting form of the Boltzmann equation for Coulomb interactions in the so-called grazing collision limit. This asymptotic regime leads to a nonlinear Fokker–Planck operator that describes the cumulative effect of many small-angle deflections, typical in plasma physics and astrophysical settings. While this limit simplifies the treatment of long-range interactions, its numerical realization often requires careful handling of singular kernels and additional regularization parameters~\cite{PareschiToscaniVillani2003}.

In the context of plasma fusion research, accurately incorporating collisional effects is increasingly recognized as essential. Collisions influence critical mechanisms such as turbulent transport, neoclassical currents, and energy deposition. High-fidelity simulation of collisional dynamics is vital for improving predictive accuracy in magnetic and inertial confinement fusion scenarios \cite{pan2021importance,izacard2025fusion}. For instance, hybrid fluid-kinetic models embedding Landau-like operators have demonstrated improved capabilities in reproducing fusion burn propagation and current profiles  \cite{izacard2025fusion}. As a result, there is growing consensus that efficient and structure-preserving collisional solvers are crucial components of next-generation simulation tools for fusion-grade plasmas.

Direct Simulation Monte Carlo (DSMC) methods have historically played a central role in the numerical approximation of collisional kinetic equations, including the spatially homogeneous Landau equation. Originally developed for the Boltzmann equation, DSMC approaches have been adapted to handle the Fokker--Planck--Landau operator by interpreting the collision term as the limit of grazing-angle Boltzmann collisions \cite{BobylevNanbu2000,DimarcoCaflischPareschi2010}. These methods simulate binary interactions between particles using probabilistic rules, resulting in a stochastic algorithm that is simple to implement, preserves conservation properties and exhibits linear computational complexity with respect to the number of particles. However, DSMC schemes typically suffer from statistical noise and may require extensive averaging to achieve accurate results. Furthermore, the grazing collision limit introduces an additional parameter, which may affect the accuracy of the approximation. Despite these limitations, DSMC remains an attractive option due to its scalability and ease of extension to complex geometries and boundary conditions. Furthermore, recently in \cite{MedagliaPareschiZanella2025}, the authors have shown how these techniques can be efficiently coupled into existing particle-in-cell (PIC) codes to introduce collisional effects in plasma physics simulations. We refer also to \cite{DuLiXieYu2025} for related stochastic particle solvers.

Alongside stochastic particle approaches, a broad class of deterministic methods has been developed by discretizing the Landau equation directly in velocity space. Early finite-difference formulations were designed to reproduce, at the discrete level, the main structural properties of the collision operator. In particular, conservative and entropy-dissipating schemes were introduced in \cite{DegondLucquinDesreux1994, BC1999}, ensuring preservation of the collisional invariants and convergence toward discrete Maxwellian equilibria. A major difficulty of these approaches is the nonlocal dependence of the diffusion and friction coefficients on the distribution function, whose direct evaluation generally entails a quadratic computational cost with respect to the number of velocity degrees of freedom.

This limitation motivated the development of fast deterministic solvers exploiting the integral and convolutional structure of the Landau operator. In \cite{BCDL1997}, conservative and entropy-consistent discretizations were combined with multipole expansions, removing the quadratic cost associated with the direct evaluation of the velocity interactions. Related fast-multipole formulations were further analyzed in \cite{Lemou1998multipole}, while implicit conservative schemes and efficient iterative solvers were subsequently investigated to overcome the restrictive stability conditions associated with the diffusive character of the equation \cite{LemouMieussens2005}. An alternative direction was introduced in \cite{pareschi2000fastlandau}, by expressing the collision operator in Fourier variables and exploiting its weighted convolution structure to evaluate it by means of fast Fourier transforms. This method reduces the complexity of a deterministic discretization from the quadratic cost of a direct implementation to an FFT-based complexity, while retaining spectral accuracy and preserving mass exactly, with momentum and energy recovered up to spectral accuracy. These ideas have provided the basis for several spectral solvers for the Landau equation and its coupling with kinetic transport models~\cite{FP2002, PareschiToscaniVillani2003, ZhangGamba2017, PennieGamba2020}. Fast Fourier-based evaluation of collision operators has also been
developed for the Boltzmann equation~\cite{MouhotPareschi2006},
combining spectral accuracy with reduced computational complexity. Recent developments include a spectral collocation formulation
for Coulomb interactions~\cite{Filbet2020Collocation} and convergence
analysis of Fourier spectral approximations of the homogeneous
Landau equation~\cite{filbet2025numerical}. Related approaches have been proposed recently in \cite{CarrilloThalhammer2026}.

More recently, a number of advances have focused on the numerical approximation of the spatially homogeneous Landau equation through deterministic particle-based methods. In particular, the work \cite{carrillo2020particle} introduces a deterministic blob particle approach based on a mollified density \cite{craig2016blob} and a regularized evaluation of the Landau collision operator, leading to a system of coupled ordinary differential equations for the particle velocities. This method benefits from its variational structure and compatibility with energy conservation and entropy production principles, but its computational cost scales quadratically with the number of particles due to the dense pairwise interactions. Treecode acceleration was also explored
in~\cite{carrillo2020particle}. An alternative strategy proposed in \cite{CarrilloJinTang2022} employs a randomized batch approximation of the particle interactions to reduce the per-step cost to linear complexity, at the expense of introducing stochastic noise. In \cite{wang2025score,ilin2025transport}, instead, the authors reconstruct the particles' vector field by employing neural networks and score matching techniques.

\begin{figure}[t]
\centering
\includegraphics[width = 1\linewidth]{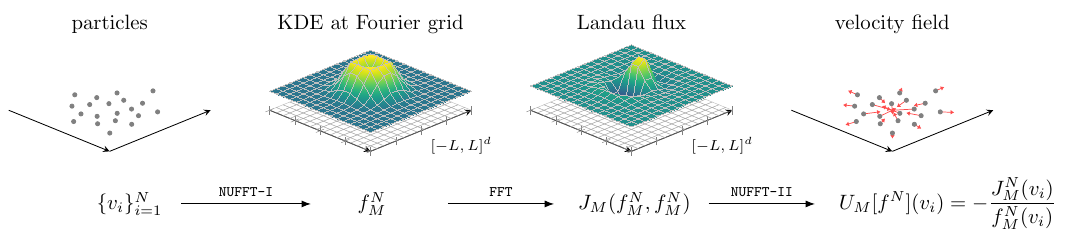}
\caption{Illustration of the fast spectral particle method for the Landau equation. 
$N$ is the number of particles, and $M$ the number of Fourier modes per dimension. The Fourier modes of the empirical measure are obtained via Non-Uniform FFT (Type I), and are used for the fast evaluation of the flux $J_M$. This corresponds to a Kernel Density Estimation (KDE) via the trigonometric polynomial $f^N_M = \P_M f^N$. 
The flux and density are then evaluated at the particle locations
using the uniform-to-nonuniform NUFFT (Type II), yielding the
velocity field $U_M[f^N]$ driving the particle evolution $\dot v_i = U_M[f^N](v_i)$. The overall computational cost is $\mathcal{O}(N + M^d \log M)$.}
\label{fig:illustration}
\end{figure}

In this work, we transfer the fast Fourier strategy underlying deterministic spectral solvers to a particle representation of the solution. More precisely, we propose a spectral particle method that combines the flexibility of particle-based solvers with the efficiency of fast spectral algorithms, exploiting the convolutional structure of the Landau flux in Fourier variables \cite{pareschi2000fastlandau}.
Given a set of $N$ particles we reconstruct the density in Fourier space via non-uniform Fourier Transform (FT), analogously to spectral PIC methods or particle-in-Fourier (PIF) methods for Vlasov--Poisson equations \cite{lehe2016pic,ameres2021particle,mitchell2019efficient,sonnen2024variational}. 
Then, we compute the particles' velocity vector field via a spectral approximation of the Landau flux and, again, non-uniform FT. See Figure \ref{fig:illustration} for an illustrative presentation of the proposed method.

Let $M$ be the number of modes per dimension, the direct non-uniform FT already scales as $\mathcal{O}(NM^d)$, and so linearly with respect to the number of particles. By using  Non-Uniform Fast Fourier Transform (NUFFT) techniques  \cite{potts2001nufft,GreengardLee2004} the cost can be reduced further to $\mathcal{O}(N|\log \epsilon|^d + M^d \log M)$ per time step. Here,  $\epsilon>0$ is the requested relative tolerance of the approximate NUFFT evaluation. 
Therefore, the Fourier-based computation of the velocity field is linear in the number of particles if compared to the blob-type particle method with cost $\mathcal{O}(N^2)$ proposed in \cite{carrillo2020particle}. While sharing the same computational complexity $\mathcal{O}(N)$ as DSMC \cite{DimarcoCaflischPareschi2010}, the accuracy achieved is higher by several orders of magnitude, see \cite{MedagliaPareschiZanella2024}.
In the present formulation the particle weights preserve total mass exactly while
momentum, energy and entropy are monitored numerically; their
preservation or dissipation is not enforced by the proposed scheme. We will leave to future research the adaptation of the techniques recently proposed in \cite{PareschiRey2022} to the current setting.

The theoretical analysis we perform (see Theorem \ref{t:main}) shows that the error introduced by spectral approximation is negligible with respect to the kernel density estimation one, allowing us to take a much smaller number of modes $M$ with respect to the number of particles $N$.  Numerical experiments over the two-dimensional  BKW benchmark test validate the analysis and showcase the efficiency and accuracy of the spectral particle method, and compare it to the blob method.


After recalling the homogeneous Landau equation, in Section \ref{sec:derivation} we derive the spectral particle method starting from its formulation as a continuity equation. In Section \ref{sec:error} we provide an error analysis for the spectral reconstruction of the velocity vector field. Section \ref{sec:num} is devoted to the numerical implementation of the method and validation against 2D benchmark problems. Final remarks and perspective research direction conclude the manuscript.

\section{A fast spectral particle method}
\label{sec:derivation}

We consider a spatially homogeneous particle distribution $f(t,v)$, $v \in\Rd$, evolving according to the Landau equation
\begin{equation} \label{eq:landau}
\frac{\partial f}{\partial t}  = Q(f,f)
\end{equation}
with 
\[Q(f,f)(t,v) :=  \nabla \cdot \int_{\Rd} A(v - v_*)\left( f(t,v_*) \nabla_v f(t,v) -  f(t,v) \nabla_{v_*} f(t,v_*) \right) dv_*\,.\]
In the above equation, $A(z)$ is the Landau tensor 
\[
A(z) = |z|^{\gamma + 2} \left( I - \frac{z \otimes z}{|z|^2} \right)\,,
\]
where $I$ is the $d$-dimensional identity matrix, and $\otimes$ the outer product. The parameter $\gamma$ satisfies $-d -1 \leq \gamma \leq 1$ with $d\geq 2$. The most relevant case is the so-called Coulomb one where $\gamma = -3$, $d = 3$, which describes the particles' collisions in plasmas.

By integrating $Q(f,f)$ against the test function $\phi(v) = 1, v, |v|^2$, one can see that solutions to \eqref{eq:landau} preserve mass, momentum, and energy. With $\phi(v) = \log f(v)$, instead, it is possible to prove the decay of the entropy $\int \log(f) f$. The equilibrium and asymptotic state is given by the Maxwellian distribution 
\begin{equation} \label{eq:maxwellian}
\mathcal{M}(v) = \frac{\rho}{(2 \pi T)^{d/2}} \exp\left(  -\frac{|v - u|^2}{2T}\right) 
\end{equation}
with $\rho, T>0$ and $u\in \Rd$ constants dependent on the distribution at $t = 0$.

\subsection{Deterministic particle approximation}

The derivation of a particle method for \eqref{eq:landau} starts by noting that the equation can be written, see \cite{carrillo2020particle},  as a  continuity equation
\begin{equation}
\label{eq:continuity}
\frac{\partial f}{\partial t} + \nabla \cdot \left( U[f]\, f \right) = 0
\end{equation}
with nonlinear velocity vector field given by
\begin{equation}
U[f](v) :=  -  \int_{\Rd} A(v - v_*) \left( \frac{\nabla f(v)}{f(v)} - \frac{\nabla f(v_*)}{f(v_*)}  \right) f(v_*)\, dv_*\,.
\label{eq:U}
\end{equation}
Next we approximate $f$ via an empirical probability distribution 
\begin{equation}
\label{eq:fN}
f^N(t,v) = \sum_{i = 1}^N w_i \delta(v - v_{i}(t))
\end{equation}
associated with a set of $N$ particles $v_i(t) \in \Rd, i = 1,\dots,N$ with fixed, but possibly different, weights $w_i>0, i = 1,\dots, N$.  

It is important to note now that the velocity field $U[f^N]$ is not well defined as the empirical measure comprises Dirac delta distributions which do not admit a density. To evaluate the velocity field, therefore, we reconstruct $f$ via kernel density estimation (KDE). Let $\psi$ be a sufficiently regular reconstruction kernel of unit
mass. The reconstructed density is given by
\begin{equation} \label{eq:fNve}
\tilde{f}^N(t,v) = \sum_{i=1}^N w_i \psi(v - v_i(t))\,.
\end{equation}
Depending on how well $\psi$ approximates $\delta$,  the reconstructed density $\tilde{f}^N$ can be considered a good approximation of the original distribution $f$, such that, most importantly, 
\[ U[\tilde{f}^N]\approx U[f]\,.\]

Typical examples of $\psi$ are the Gaussian kernel and B-splines $\psi = \psi_\ve$ with a given variance parameter $\ve>0$, but the choice can be even more general and include direct neural network approximation of $\tilde f^N$. As we will see in the next sections, we consider here $\psi = \psi_M$ to be a trigonometric polynomial arising from a Fourier truncation for $M\in \mathbb{N}$.

Given a KDE $\tilde{f}^N$, a deterministic particle scheme for the Landau equation is given then by the system of ODEs
\begin{equation} \label{eq:odes1}
\frac{d }{dt}{v}_i(t) = U[\tilde{f}^N] \left(v_i(t) \right) \qquad i = 1,\dots, N\,.
\end{equation}
We show next how to derive a fast computation of $U[\tilde{f}^N]$ by relying on the convolutional structure of the Landau flux.

\subsection{Spectral computation of the velocity field}

The spectral method for the Landau equation \cite{pareschi2000fastlandau}  is based on the fact that, after truncation to a bounded domain, the collisional operator $Q(f, f)$ can be exactly computed in terms of the Fourier modes of $f$, provided $f$ is a periodic function over a domain $[-L,L]^d$. Also, the computation is very efficient due to its convolutional structure. While sharing a similar structure, the velocity field $U[f]$ cannot be written exactly in terms of the Fourier modes due to the presence of the fractions $\nabla f(v)/ f(v)$ and  $\nabla f(v_*)/ f(v_*)$, see \eqref{eq:U}.
 To address this problem, we write the velocity field as 
\begin{equation} \label{eq:UJf}
 U[f](v) =  - \frac{J(f,f)(v)}{f(v)}
  \end{equation}
  where $J$ is the flux 
 \be \label{eq:J}
J(f,g)(v) := \int_{\mathbb{R}^d} A(v - v_*) \left[ g(v_*) \nabla f(v) - f(v) \nabla g(v_*) \right] dv_*\,
\ee
which satisfies $Q(f,f) = \nabla\cdot J(f,f)$, see \eqref{eq:landau}. As done for the collisional operator $Q(f,f)$ in \cite{pareschi2000fastlandau}, the flux can be computed in a spectrally accurate and efficient way via  FFT. We recall the main steps of the derivation. 

The first step consists of restricting the flux evaluation over the torus $\Td_L =\Rd /(2L\mathbb Z)^d$. For periodic functions $F,G\in C^1(\Td_L)$ we define
\begin{equation}\label{eq:Jper}
 J_\per(F,G)(v)
:=
\int_{[-L,L]^d}
A(q)\left[
G(v-q)\nabla F(v)-F(v)\nabla G(v-q)
\right]\,dq \,,
\end{equation}
where $v-q$ is interpreted modulo $2L\mathbb Z^d$.  Note that the integral is now with respect to the relative velocity $q=v-v_*$. As we will show in Lemma \ref{l:truncation}, if the domain length $L$ is sufficiently large with respect to the support of the particles' distribution, integrating over $\Td_L$ or over $\Rd$ is equivalent.

Let $\P_M: L^2(\Td_L)\to L^2(\Td_L)$ denote the projection into the space of trigonometric polynomials of degree at most $M/2\in \mathbb{N}$; we have
  \[
\P_M f (v) := \sum_{k \in \llbracket -M/2,M/2-1\rrbracket^d} \hat{f}_k e^{i \pi k \cdot v / L},
\quad
\hat{f}_k = \frac{1}{(2L)^d} \int_{[-L,L]^d} f(v) e^{-i \pi k \cdot v / L} dv,
\]
where we used the notation $\llbracket a, b \rrbracket^d = [a,b]^d \cap \mathbb{Z}^d$  for the multidimensional summation along each component of the integer vector $k$. 
For given sufficiently smooth periodic  functions $F,G$, we consider the truncated flux
\begin{equation}\label{eq:JM}
J_M(F,G)
:=
\P_M J_\per ( F, G).
\end{equation}

Thanks to this approximation, the modes of $J_M(\P_M f, \P_M f)$ can be computed via the modes of $f$  very similarly
as in the spectral method for the collisional operator \cite{pareschi2000fastlandau}. We work out the details in the following.

By plugging the definition of $\P_M f$ into \eqref{eq:Jper}
\begin{align*}
J_{\mathrm{per}}(\P_M f,&\P_M f)(v)
= \int A(q) \Bigg[
\frac{i\pi}L \sum_{m,\ell} \hat{f}_m \hat{f}_l \exp\left(\frac{i\pi}L ( l \cdot v + m\cdot(v-q))\right) l
  \\
& \hspace{3cm}
- \frac{i\pi}L  \sum_{m,\ell} \hat{f}_m \hat{f}_l \exp\left(\frac{i\pi}L ( l \cdot v + m\cdot(v-q))\right) m
\Bigg] dq \\
& =   \int A(q) 
\frac{i\pi}L  \sum_{m,\ell} \hat{f}_m \hat{f}_l \exp\left(\frac{i\pi}L ( l \cdot v + m\cdot(v-q))\right) (l- m) dq \\
&  = \sum_{m,\ell} \hat{f}_m \hat{f}_l \int \frac{i\pi}LA(q) \exp\left(\frac{i\pi}L ( l \cdot v + m\cdot(v-q))\right) (l- m) dq \\
&  = \sum_{m,l} \hat{f}_m \hat{f}_l  \left(\frac{i\pi}L(l - m) \int A(q) \exp\left(-\frac{i\pi}L  m\cdot q \right) dq\right) \exp\left(\frac{i \pi}L  (m+l)\cdot v \right) \\
& = \sum_{m,l} \hat{f}_m \hat{f}_l \hat{\beta}(l,m) \exp\left(\frac{i \pi}L  (m+l)\cdot v \right)
\end{align*}
where the spectral kernel vector $\hat \beta(l,m)$ can be written as
\[
\hat \beta(l,m) : = \frac{i \pi}L (l - m) \int A(q)e^{-i\pi m \cdot q/L} dq\,.
\]
From $A(q) = \Psi(q) (I  - \mu\mu^\top)$, with $\mu = q/|q|$, $\Psi(q) = |q|^{\gamma + 2}$ we have
\begin{align*}
\hat \beta(l,m) 
& =   \frac{i \pi}L (l - m) \int \Psi(q) (I  - \mu\mu^\top) e^{-i\pi m \cdot q/L} dq \\
& =   \frac{i \pi}L 
\begin{pmatrix}
l_1 - m_1 \\
l_2 - m_2
\end{pmatrix}
\! \int \! \Psi(q)  e^{-i\pi m \cdot q/L} dq 
 \\ & \qquad \qquad
- 
\begin{pmatrix}
l_1 - m_1 \\
l_2 - m_2
\end{pmatrix}
\cdot \frac{i \pi}L  \!\int \!\Psi(q) 
\begin{pmatrix}
\mu_{1}^2 & \mu_{1}\mu_2 \\
\mu_{1}\mu_2 & \mu_2^2 
\end{pmatrix}
 e^{-i\pi m \cdot q/L} dq.
\end{align*}

Now, if we define 
\[
D(m) :=\frac{i\pi}L \int \Psi(q) e^{-\frac{i\pi}L m \cdot q}dq  \qquad \textup{and}\qquad 
I_{pq}(m) :=  \frac{i\pi}L\int \Psi(q) \mu_{p}\mu_q e^{-\frac{i\pi}L m\cdot q}  dq \,,
\]
we get $\hat \beta (l,m) =  (l - m) D(m) - (l - m) I(m)$.
Applying the projection $\P_M$ and setting $\ell=k-m$, we obtain,
for each retained mode $k$,
\begin{align}
(\hat J_M)_{k}(\P_M f, \P_M f) &= \sum_{m}  \hat f_m \hat f_{k-m} \hat \beta (k - m, m) \notag \\
& = \sum_{m}  \hat f_m \hat f_{k-m} \left(
(k - 2m) D(m)  -  (k - 2m) I(m) \right)  \label{eq:JMk}
 \\
& = \sum_{m}  \hat f_m \hat f_{k-m} \left(
(k - m) D(m)  - (k - m) I(m)  \right) \\
& \qquad \qquad
 - \sum_{m} \hat f_m \hat f_{k-m} m\left( D(m) - I(m) \right)\,. \notag
\end{align}
Here the sums run over indices $m$ such that both $m$ and $k-m$
belong to the retained Fourier index set.
The above expression can be efficiently computed via FFT once the modes $\hat f_k$ and $D, I_{pq}$ are available. Note that $D$ and $I_{pq}$ can be computed once at the beginning of the computation, and correspond to the same kernel modes used in the spectral methods for the Landau equation \cite{pareschi2000fastlandau}.

\begin{remark}
Other choices to obtain a spectral approximation of $U[f]$ are possible. For instance, if a spectral approximation of $g = \nabla f/f$ is available, then one can 
separate the vector field into a drift and a diffusion term as 
\[U[f](v) =  - \int A(q) f(v - q) g(v) dq + \int A(q) \nabla f(v - q) dq\,,\]
which can also be computed via FFT. We found this strategy to be less stable than approximating $U[f]$ as in \eqref{eq:UJf}, possibly because a spectral approximation of $g = \nabla f/f$ introduces an error upstream in the computations.
\end{remark}

\subsection{Algorithm via non-uniform FFT}

We now illustrate the fast spectral particle scheme which combines the particles' evolution \eqref{eq:odes1} with the spectral velocity field evaluation. 
As a KDE for the particles' distribution, we will directly consider the truncated Fourier expansion $\P_M f^N$ of the empirical measure $f^N$. 
Given a truncation floor $\underline{f}>0$, the approximated velocity vector field at a particle's location $v_i$ is given by 
\begin{equation} \label{eq:UM}
U_M[f^N](v_i) := - \frac{J_M(\P_M f^N ,\P_M f^N)(v_i)}{\max\{\P_M f^N(v_i), \underline{f}\}} \,.
\end{equation}
Since $\P_M f^N$ may take non-positive values, the truncation ensures that $U_M$ is well-defined.

We are led to the system of ODEs
\begin{equation} \label{eq:odes2}
\frac{d }{dt}{v}_i(t) = U_M[f^N] \left(v_i(t) \right) \qquad i = 1,\dots, N\,,
\end{equation}
which can be interpreted as a discrete version of the nonlinear characteristic flow associated with the Landau equation in divergence form, where each particle evolves under the influence of the spectrally reconstructed velocity vector field $U_M[f^N]$. This aligns with the semi-Lagrangian perspective typically used in deterministic kinetic solvers.

As mentioned in the previous section, once the modes ${(\hat f^{N})}_k$ are available, the computation of the flux modes $\hat J_k^M$ can be obtained via FFT with a cost $\mathcal{O}(M^d \log M)$. To compute the modes of the empirical density $\P_M f^{N}$, we use the Non-Uniform FFT (NUFFT) Type-I corresponding to the nonuniform-to-uniform operation \cite{GreengardLee2004}. To evaluate the reconstructed flux and the reconstructed density at the particles' locations $v_i$ we use the uniform-to-nonuniform operation, NUFFT Type-II.  (Type-III corresponds to the case nonuniform-to-nonuniform which is not required in the proposed scheme.)
The computational cost of both NUFFT-I and NUFFT-II is $\mathcal{O}(N |\log \epsilon|^d + M^d \log M)$ where $\epsilon>0$ is a tolerance parameter of NUFFT algorithms \cite{potts2001nufft,GreengardLee2004}. Therefore, the overall computational cost in terms of $N,M$ for one time integration step of \eqref{eq:odes2} is
\[
\mathcal{O}(N + M^d \log M)\,.
\]
Algorithm \ref{alg} summarizes the main steps of the spectral particle method implemented with explicit Euler time integration, and highlights the computational complexity per step. The computational domain size $L$, the NUFFT tolerance $\epsilon$,
and the denominator floor $\underline{f}$ are additional parameters
of the algorithm. For Runge--Kutta time integration, the density
reconstruction and flux evaluation are repeated at each stage;
a fixed number of stages does not change the asymptotic complexity.

\SetKwComment{MyComment}{}{}      

\begin{algorithm}[t]
\caption{Fast spectral particle method}
\label{alg}
\hspace{-3mm}
\DontPrintSemicolon

Particles $\{v_i(0)\}_{i=1}^N \subset [-L,L]^d$ with weights $\{w_i\}_{i=1}^N$ 
 \MyComment*[r]{input}
 
Modes per dimension $M$

Time step $\Delta t >0$


Precompute 
Landau kernel coefficients \(D, I\)  \MyComment*[r]{initialization}

$t = 0$

\While{$t < t_{\max} $
}{
Compute $\hat f ^N$ \eqref{eq:fNk} via NUFFT-I 
 \MyComment*[r]{$N + M^d \log M$}
 
 Compute $ \hat J_{M}(\P_M f^N, \P_M f^N )$  \eqref{eq:JMk} via FFT
  \MyComment*[r]{$M^d \log M$}

Compute  $U_M[f^N](v_i(t))$ via \eqref{eq:UM} and NUFFT-II
 \MyComment*[r]{$N + M^d \log M$}
 
Particles' time evolution: for $i = 1,\dots,N$    \MyComment*[r]{$N$}
\vspace{2mm}
$\qquad  v_i(t+\Delta t ) = v_i(t) + \Delta t U_M[f^N] (v_i(t))$

\vspace{2mm}
      $t = t + \Delta t$   \MyComment*[r]{total cost:  $N + M^d \log M$ }
    }
 
  \end{algorithm}

Note that, if needed, the reconstructed density can be regularized by applying a spectral filter to the coefficients $(\hat f ^N)_k$, without affecting the computational cost. Filters damp high-frequency components and therefore may reduce aliasing and improve numerical stability. 
It is interesting to note that $f^N_M = \P_M f^N$ can be written in the form of \eqref{eq:fNve} with $\psi(x) = \psi_M(x) = 1/(2L)^d\sum_k \exp(i\pi k\cdot x/L)$, $k \in  \llbracket -M/2,M/2-1\rrbracket^d$, a Dirichlet periodic polynomial. Indeed, since the Fourier coefficients are given by 
\begin{equation}\label{eq:fNk}
\hat f^N_k = \frac 1{(2L)^d} \sum_{i=1}^N w_i e^{- i \pi k\cdot v_i/L }\,, \qquad k \in  \llbracket -M/2,M/2-1\rrbracket^d
\end{equation}
then
\begin{equation} \label{eq:fNM}
\begin{aligned} 
f_M^N(v)
= P_M f^N(v) 
 = \sum_k \hat{f}_k^N
   e^{i\pi k\cdot v/L} &= \frac{1}{(2L)^d}
   \sum_{k }\sum_{i=1}^N
   w_i e^{-i\pi k\cdot v_i/L}
       e^{i\pi k\cdot v/L} \\
&= \sum_{i=1}^N w_i
   \left(
      \frac{1}{(2L)^d}
      \sum_{k}
      e^{i\pi k\cdot (v-v_i)/L}
   \right) .
\end{aligned}\end{equation}
While $\psi_M$ has mass one, $\int_{[-L,L]^d}\psi_M  = 1$, and so $f^N_M$ preserves the mass, it might take negative values.

\begin{remark}
Directly computing the Fourier coefficients  of $f^N$ via NUFFT corresponds in the literature to the Fourier-in-cell (FIC) method \cite{mitchell2019efficient,ameres2018stochastic}.  The alternative approach is the spectral PIC  method \cite{lehe2016pic,sonnen2024variational} where one first computes the value of a mollified empirical density $\tilde{f}^N(\overline{v}_k)$, $k \in \llbracket -M/2,M/2-1 \rrbracket^d$  at the uniform Fourier grid $\overline{v}_k \in [-L,L]^d$ and then performs the usual equispaced FFT.  As noted in \cite{ameres2021particle}, the spectral PIC method can actually well approximate the grid-less FIC approach. Also, some NUFFT algorithms may include an interpolation step on the Fourier grid, making the two approaches more similar in practice, see, e.g., the discussion in  \cite{barnett2019parallel}. In the methods cited above, the particle-to-Fourier
mapping acts on particle positions in physical space to compute
the self-consistent fields. Here, the same computational principle
is applied in velocity space: the Fourier modes are reconstructed
from the particle velocities and used to evaluate the collisional
velocity field driving their evolution.    
\end{remark}



\section{Error analysis}
\label{sec:error}

The central feature of the algorithm is the approximation of the velocity field $U[f]$ with the spectrally reconstructed $U_M[f^N]$ \eqref{eq:UM}. We investigate in this section the consistency of the approximation by studying the error in terms of the number of modes $M$ and the distance between $f^N_M$ and $f$. We will consider any interaction potential with $\gamma$ in the range $-d-1\leq \gamma \leq 1$.

Throughout, for a domain $\Omega\subset \mathbb R^d$,
$C^m(\Omega)$ is the set of functions, possibly vector-valued, differentiable up to the order $m\in \mathbb{N}\cup \{0\}$ and with continuous derivatives. We will also use
\[
\|h\|_{C^0(\Omega)}:=\sup_{v\in \Omega}|h(v)|,
\qquad
\|h\|_{C^1(\Omega)}:=\|h\|_{C^0(\Omega)}+\|\nabla h\|_{C^0(\Omega)}.
\]
For $s\geq 0$ and a periodic function
$h\in L^2(\mathbb T_L^d)$ with Fourier coefficients
$\widehat h_k$, $k\in\mathbb Z^d$, we define
\[
\|h\|_{H^s(\mathbb T_L^d)}^2
:=
(2L)^d
\sum_{k\in\mathbb Z^d}
\left(
1+\frac{\pi^2|k|^2}{L^2}
\right)^s
|\widehat h_k|^2.
\]
If $\supp(h)\subset [-L,L]^d$, we identify it with its periodic extension and set $\|h\|_{H^s(\Rd)}:=\|h_\per\|_{H^s(\Td_L)}$. When the domain is clear, we simply write
$\|h\|_{H^s}$ and $\|h\|_{C^k}$. 
The same notation is used for vector-valued functions.

\begin{theorem}\label{t:main}
Consider $\gamma$ such that $-d-1 \leq \gamma \leq 1$. 
Let $f\in C^\infty(\mathbb R^d)$ be supported in $B(0,R)$ and  $f_M^N:=\P_M f^N\in C^\infty(\Td_L)$.
For $\kappa>0$, define
\[
\Omega_\kappa
:=
\{v\in \Rd \, |\,  f(v)>\kappa\}.
\]
Let $L>2R$, $s>d/2+1$, and the floor parameter $\underline{f}>0$ be such that  $\underline f\leq\kappa/2$. 

If
$ \|f_M^N-f\|_{C^0(\Td_L)}
\leq \kappa/2$ and
$
\|f_M^N-f\|_{H^{s}(\Td_L)}
\leq 1,
$
there exists a constant $C = C(\gamma,d,R,L,s,\|f\|_{H^s})>0$, such that
\[
\|U_M[f_M^N]-U[f]\|_{C^0(\Omega_\kappa)}
\leq
C \left(
\frac{1}{\kappa}
M^{-(s-d/2-1)}
+
\frac{1}{\kappa^2}
\|f_M^N-f\|_{C^1(\Td_L)} \right) .
\]
\end{theorem}

The analysis shows that the approximation of $U$ with $U_M$ is spectrally accurate, and therefore we expect the error to be dominated by the reconstruction of the particle density $f$ from the empirical distribution $f^N$. This is done here via the trigonometric polynomials $f^N_M = \P_M f^N$, but additional filtering or different particle shapes can be considered without affecting the error decomposition (see Remark \ref{rmk:eps}).

In the numerical experiments we perform,  particles are initialized on a Cartesian grid of mesh size $h$, with weights
$
w_i\propto f(v_i)h^d, \sum_i w_i = 1,
$ such that 
$f^N=\sum_{i=1}^N w_i\delta_{v_i}$.
If the particle grid is sufficiently fine with respect to the standard Fourier grid, we expect $\P_M f^N \approx f^M = \P_M f$. 
At time $t=0$, the reconstruction error can be decomposed as
\[
\|f_M^N-f\|_{C^1(\mathbb{R}_L^d)}
\le
\|P_M(f^N-f)\|_{C^1(\mathbb{R}_L^d)}
+
C M^{-(s-d/2-1)}\|f\|_{H^s(\mathbb{R}_L^d)}.
\]
The first term accounts for the particle quadrature error in the
retained Fourier coefficients. If this term is sufficiently small,
the reconstruction inherits the spectral accuracy of the Fourier
projection. This initial accuracy may deteriorate at positive times
due to the distortion induced by the particle flow~\cite{cottet1984particle}.
This effect can be mitigated, for instance, by particle remappings   \cite{denavit1972numerical,crouseilles2009forward} or linear transformation of the particles' shapes \cite{cohen2000optimal,pinto2016uniform}. A detailed discussion on the accuracy of Fourier density reconstruction when using Monte Carlo initialization of the particles can be found in \cite[Chapter 3]{ameres2018stochastic}.

An important caveat of Theorem \ref{t:main} is the restriction to the domain $\Omega_\kappa$ where the density is strictly positive $f(v) > \kappa>0$ and the accuracy of $U_M$ is poor near vacuum. 
However, it is important to remark that the numerical method evaluates the velocity field only at particle
locations, with the denominator floor ensuring that it remains
well-defined even when the reconstructed density is nonpositive.


Finally, estimates for non-compactly supported densities may be derived following the strategy proposed in  \cite{filbet2025numerical} for the convergence analysis for spectral methods for the Landau equation. The full convergence analysis of the blob particle method for a regularized Landau equation, instead, has been proposed in \cite{carrillo2023convergence} via a coupling method and the $L^\infty$-Wasserstein metric.

\begin{remark}\label{rmk:eps}
Particle smoothing or Fourier filtering can be incorporated in the method to mitigate numerical instability and aliasing effects \cite{sonnen2024variational}, without affecting the error analysis. In the case of B-splines or Gaussian filtering, for instance, this corresponds to substituting $f_M^N = \P_M(f^N)$ with $f^{N,\ve}_M : = \P_M( \psi_\ve \ast f^N)$ where $\psi_\varepsilon(v)=\varepsilon^{-d/2}\psi(v/\sqrt{\varepsilon})$ is the mollifier. Error estimates under such smoothing in terms of the number of particles $N$ and the parameter $\ve>0$ can be found in \cite{kim2019uniform} for Monte Carlo initialization, and in \cite{anderson1985,ala2007corrective} for grid initialization. We will also investigate the effect of filtering numerically in Section \ref{sec:num} in a benchmark test.  
\end{remark}

\subsection{Error due to spectral approximation of the flux}

To prove Theorem \ref{t:main}, we first investigate the error introduced by approximating $J$ \eqref{eq:J} with $J_{M}$ \eqref{eq:JM}.
As in Theorem's assumptions, we  consider compactly supported data for which there are no truncation errors or aliasing effects, provided the truncation domain $[-L,L]^d$ is sufficiently large.

\begin{lemma}\label{l:truncation}
Let $f,g \in C^1(\Rd)$ be compactly supported over $B(0,R)$. If $L>2R$ then
\[
J(f,g)(v) = J_\per(f_\per,g_\per)(v) \quad \textup{for all} \quad v\in [-L, L]^d
\]
where  $f_\per, g_\per \in C^{1}(\Td_L)$ are their corresponding periodic extensions over the torus $\Td_L$.
\end{lemma}
\begin{proof}
Since for any $v\in [-L,L]^d\setminus B(0,R)$ we have $f(v) = 0$ and $\nabla f(v) = 0$ , it holds $J(f,g) (v) = 0$. For the same reason we have $J_\per(f_\per, g_\per)(v) = 0$ for $v\in  [-L,L]^d\setminus B(0,R)$. 
For any point $v\in B(0,R)$, instead, we need to make sure that the periodic flux does not integrate over neighboring copies of $\supp(g)$. Consider the copies of the ball $B(0,R)$ given by
\[
B_\ell:=B(0,R)+2L\ell,\qquad \ell\in\mathbb Z^d.
\]
Suppose that $v-q\in B_\ell$ for some $\ell\neq0$ and $q\in[-L,L]^d$. Then, there exists $y\in B(0,R)$ such that
$
q=v-y-2L\ell.
$
Since $\|\ell\|_\infty\geq 1$, by the reverse triangle inequality we have
\[
\| q\|_\infty \geq 2L - \|v-y \|_\infty \geq 2L - 2R > L
\]
where the last inequality follows from $L>2R$. This contradicts $q\in[-L,L]^d$,  and therefore $J(f,g)(v) = J_\per(f_\per,g_\per)(v)$ for $v\in B(0,R)$ too.
\end{proof}

Next, we collect some stability estimates for $J_\per$.

\begin{lemma}\label{l:flux_stability}
Let $-d-1 \leq \gamma \leq 1$ and $f, g\in C^\infty(\Td_L)$.  There exist  $C(\gamma,d,L)$, $C(\gamma,d,L,s)>0$ such that
\begin{align*}
\|J_\per(f,g)\|_{C^0(\Td_L)} & \le\ C(\gamma,d, L)\,\|g\|_{C^1(\Td_L)}\,\|f\|_{C^1(\Td_L)} \\
\|J_\per(f,g)\|_{H^s(\Td_L)} &  \le\ C(\gamma,d,L,s)\,\|g\|_{H^{s+1}(\Td_L)}\,\|f\|_{H^{s+1}(\Td_L)} \,\qquad \textup{for} \;\; s>d/2\,.
\end{align*}

\end{lemma}

\begin{proof} In the proof (and in subsequent ones) the constant $C>0$ may take different values, from line to line. We sometimes stress the dependencies as $C = C(\gamma,d,L)$.

By definition of $A$ we have $\| A(q)\|_2 = |q|^{\gamma+2}$ for $q \neq 0$. Therefore, since the integration domain is bounded and $\gamma> - d -2$, it holds
\[
\int_{[-L,L]^d} \|A(q)\|_2 dq  = \int_{[-L,L]^d} |q|^{\gamma+2} dq < C(\gamma,d,L)\,.
\]
Direct computations then lead to the first estimate
\begin{align*}
|J_\per(f,g) (v) | & \leq \int_{[-L,L]^d} \|A(q) \|_2 \left(  |g(v - q)| |\nabla f(v)|  + |f(v)||\nabla g (v -q) |\right) dq \\
& \leq C(\gamma,d,L) \left(\| g\|_{C^0(\Td_L)}\| \nabla f \|_{C^0(\Td_L)} + \| f\|_{C^0(\Td_L)}\| \nabla g\|_{C^0(\Td_L)} \right)  \\
& \leq C(\gamma,d,L) \|g\|_{C^1(\Td_L)}\,\|f\|_{C^1(\Td_L)} \,.
\end{align*}

Note that, thanks to the integrability of \(A(q)\) (and of its components \(A_{ij}\)) and Young’s inequality for periodic convolutions, we also have
\[
\|A*\nabla g\|_{H^s(\mathbb T_L^d)}
\leq C(\gamma,d,L)\|g\|_{H^{s+1}(\mathbb T_L^d)}
\]
and
\[
\|A_{ij}*g\|_{H^s(\mathbb T_L^d)}
\leq C(\gamma,d,L)\|g\|_{H^s(\mathbb T_L^d)}.
\]
For \(u,v\in H^s(\mathbb T_L^d)\), it holds
\[
\|uv\|_{H^s(\mathbb T_L^d)}
\leq
C(d,L,s)
\|u\|_{H^s(\mathbb T_L^d)}
\|v\|_{H^s(\mathbb T_L^d)},
\]
since \(H^s(\mathbb T_L^d)\) is a Banach algebra for \(s>d/2\) \cite[Theorem 4.39]{adams2003sobolev}.

Using these estimates, we obtain
\[
\begin{aligned}
\|(A_{ij}*g)\partial_jf\|_{H^s(\mathbb T_L^d)}
&\leq C(d,L,s)
\|A_{ij}*g\|_{H^s(\mathbb T_L^d)}
\|\partial_jf\|_{H^s(\mathbb T_L^d)}\\
&\leq C(\gamma,d,L,s)
\|g\|_{H^s(\mathbb T_L^d)}
\|\partial_jf\|_{H^s(\mathbb T_L^d)},
\end{aligned}
\]
and, analogously,
\[
\|f(A_{ij}*\partial_jg)\|_{H^s(\mathbb T_L^d)}
\leq C(\gamma,d,L,s)
\|f\|_{H^s(\mathbb T_L^d)}
\|\partial_jg\|_{H^s(\mathbb T_L^d)}.
\]
Summing over \(i,j=1,\dots,d\) yields the desired estimate in terms of the norms \(\|g\|_{H^{s+1}(\mathbb T_L^d)}\) and \(\|f\|_{H^{s+1}(\mathbb T_L^d)}\).
\end{proof}

We recall and collect some fundamental results 
regarding Fourier truncation errors. For any $f\in H^{s}(\Td_L)$, we have \cite[Theorem 2.1]{shen2011spectral}
\begin{equation} \label{eq:err_Hs}
\| f - \P_M f\|_{H^r(\Td_L)} \leq C(L,d,s) M^{r - s} \| f\|_{H^s(\Td_L)}\qquad \textup{for any} \;\; 0<r<s\,.
\end{equation}

\begin{lemma}
For \(f\in H^s(\Td_L)\) it holds:
\begin{align}\label{eq:err_C0}
    \|f-\P_Mf\|_{C^0(\Td_L)}
    &\leq
    C(L,s,d)M^{-(s-d/2)}
    \|f\|_{H^s(\Td_L)} &\textup{if}\quad  s>d/2\,, \\
\label{eq:err_C1}
    \|f-\P_Mf\|_{C^1(\Td_L)}
    &\leq
    C(L,s,d)M^{-(s-1-d/2)}
    \|f\|_{H^s(\Td_L)}  &\textup{if} \quad s>d/2+1\,.
\end{align}
\end{lemma}

\begin{proof} We follow similar computations to
\cite[Lemma~6.3]{potts2021fourier}, where more general weighted Fourier truncation estimates are proved. Let $K_M :=\llbracket -M/2,M/2-1\rrbracket^d$. 
Since the Fourier basis functions have unit modulus, we have
\[
    \|f(v)-\P_Mf(v)\|_{C^0(\Td_L)}
= \left \|
    \sum_{k\notin K_M}
    \hat f_k
    \exp\left(\frac{i\pi k\cdot v}{L}\right)\right\|_{C^0(\Td_L)} \leq \sum_{k\notin K_M}
    |\hat f_k|\,.
\]
 By the Cauchy--Schwarz inequality we have
\begin{align*}
    \sum_{k\notin K_M}
    |\hat f_k|
    &=
    \sum_{k\notin K_M}
    \left(
        1+\frac{\pi^2|k|^2}{L^2}
    \right)^{-s/2}
    \left(
        1+\frac{\pi^2|k|^2}{L^2}
    \right)^{s/2}
    |\hat f_k|
    \\
    &\leq
    \left(
        \sum_{k\notin K_M}
        \left(
            1+\frac{\pi^2|k|^2}{L^2}
        \right)^{-s}
    \right)^{1/2}
    \left(
        \sum_{k\notin K_M}
        \left(
            1+\frac{\pi^2|k|^2}{L^2}
        \right)^s
        |\hat f_k|^2
    \right)^{1/2}.
\end{align*}
For $s>d/2$, the first factor can be bounded as
\[
    \sum_{k\notin K_M}
    \left(
        1+\frac{\pi^2|k|^2}{L^2}
    \right)^{-s}
    \leq
    C(L,s,d)M^{d-2s},
\]
see proof of \cite[Theorem 3.1]{gorner2012efficient}. The definition of the periodic Sobolev norm gives the bound for the second factor
\[
    \left(
        \sum_{k\notin K_M}
        \left(
            1+\frac{\pi^2|k|^2}{L^2}
        \right)^s
        |\hat f_k|^2
    \right)^{1/2}
    \leq
    (2L)^{-d/2}\|f\|_{H^s(\Td_L)}.
\]
The two estimates together lead to \eqref{eq:err_C0}. 

The same argument leads to the estimate for the $C^1$ norm \eqref{eq:err_C1}. The only difference is that, when taking the derivative, we introduce an additional factor of $|k|$ in the Fourier sum, reducing the decay order in $M$ from $s-d/2$ to $s-d/2-1$.
\end{proof}

With the derived error estimates we show the spectral accuracy of the flux approximation.

\begin{proposition}\label{p:J_error}
Let $-d-1 \leq \gamma \leq 1$ and  $f\in C^\infty(\Rd)$ be compactly supported in $B(0,R)$.
Fix $L>2R$. 
If $s>d/2+1$, then there exists $C(\gamma,d,R,L,s)>0$ such that
\[
\|J_M(\P_M f,\P_M f)-J(f,f)\|_{C^0([-L,L]^d)}
\le
C(\gamma,d,R,L,s)
M^{-(s-d/2-1)}
\|f\|_{H^s(\mathbb R^d)}^2\, .
\]

\end{proposition}

\begin{proof}
As before, we denote by $f_{\per}\in C^\infty(\Td_L)$ the periodic extension of $f$ so that we have $f^M=\P_M f_{\per}.$
We split the error as
\begin{align*}
\|J_M(f^M,f^M)-J(f,f)\|_{C^0([-L,L]^d)}
&\leq 
\|\P_M J_{\per}(f^M,f^M)-J_{\per}(f^M,f^M)\|_{C^0(\Td_L)} \\
& \qquad+ \|J_{\per}(f^M,f^M)-J_{\per}(f_{\per},f_{\per})\|_{C^0(\Td_L)}  
\\
&  \qquad + \|J_{\per}(f_{\per},f_{\per}) -J(f,f)\|_{C^0([-L,L]^d)}  =: I_1 + I_2 + I_3\,.
\end{align*}
Since $L>2R$, by Lemma~\ref{l:truncation} we have $I_3 = 0$.
For the term $I_1$, we apply the projection estimate \eqref{eq:err_C0} and Lemma~\ref{l:flux_stability} with exponent $s-1>d/2$
\begin{align*}
I_1
&\leq
C(L,d,s)M^{-(s-1-d/2)}
\|J_{\per}(f^M,f^M)\|_{H^{s-1}(\Td_L)}  \\
& \leq
C(\gamma,d,L,s)M^{-(s-1-d/2)} \|f^M\|_{H^s(\Td_L)}^2 \,.
\end{align*}

Since $\P_M$ is the Fourier projection, it is bounded in $H^s(\Td_L)$, and therefore
$
\|f^M\|_{H^s(\Td_L)}
=
\|\P_M f_{\per}\|_{H^s(\Td_L)}
\le
\|f \|_{H^s(\Rd)} .
$
Consequently,
\[
I_1
\le
C(\gamma, d,L,s)M^{-(s-d/2-1)}
\|f\|_{H^s(\Rd)}^2 .
\]

To estimate the term $I_2$, we start by using the bilinearity of $J_{\per}$ to bound it as
\begin{align*}
I_2 &  =  
\| J_{\per}(f^M,f^M) \pm  J_{\per}(f_\per,f^M)  -J_{\per}(f_{\per},f_{\per})\|_{C^0(\Td_L)}
\\
& \leq  
\| J_{\per}(f^M-f_{\per},f^M) \|_{C^0(\Td_L)}
+
\| J_{\per}(f_{\per},f^M-f_{\per})  \|_{C^0(\Td_L)} \\
&\le
C(\gamma, d,L)
\|f^M-f_{\per}\|_{C^1(\Td_L)}
\left(
\|f^M\|_{C^1(\Td_L)}
+
\|f_{\per}\|_{C^1(\Td_L)}
\right)\,,
\end{align*}
where in the last step we used the $C^0$ stability estimate in Lemma~\ref{l:flux_stability}.
By the projection estimate \eqref{eq:err_C1},
\[
\|f^M-f_{\per}\|_{C^1(\Td_L)}
=
\|\P_M f_{\per}-f_{\per}\|_{C^1(\Td_L)}
\le
C(L,d,s)
M^{-(s-1-d/2)}
\|f\|_{H^s(\Td_L)}\,  .
\]
Since $s>d/2+1$, by the Sobolev embedding theorem we also have 
$
\|f_{\per}\|_{C^1(\Td_L)}
\le
C(L,d,s)\|f_{\per}\|_{H^s(\Td_L)}
$, 
and, as a consequence, 
$
\|f^M\|_{C^1(\Td_L)}
\le
\|f_{\per}\|_{C^1(\Td_L)}
+
\|f^M-f_{\per}\|_{C^1(\Td_L)}
\le
C(L,d,s)\|f\|_{H^s(\Td_L)}
$
for $M\ge 1$. 
Therefore, also for the term $I_2$ we obtained the desired estimate
\[
I_2
\le
C(\gamma, d,L,s)
M^{-(s-d/2-1)}
\|f_{\per}\|_{H^s(\Td_L)}^2\, .
\]
\end{proof}

\subsection{Error in the velocity field}

Finally, we apply the above results to quantify the error between
$U_M[f^N]$ and $U[f]$. 

\begin{proof}[Proof of Theorem \ref{t:main}]
Since $\supp(f) \subset [-L,L]^d$, $f = f_\per$, and we will use them interchangeably.  By assumption,
$
\|f_M^N-f \|_{C^0(\mathbb T_L^d)}
\leq  \kappa /2,
$ and for every
$v\in\Omega_\kappa$ we have
\[
f_M^N(v)
\geq
f(v)-|f_M^N(v)-f(v)|
>
\frac{\kappa}{2}.
\]
Since $\underline f\leq\kappa/2$, it follows that $\max\{f_M^N(v),\underline f\}=f_M^N(v)$, that is,  the flooring operation is inactive for every $v\in\Omega_\kappa$.

By using the above lower bounds, we have
\begin{align*}
|U_M[f_M^N](v)-U[f](v)|
&=
\left|
\frac{J_M(f_M^N,f_M^N)(v)}{f_M^N(v)}
\pm 
\frac{J(f,f)(v)}{f_M^N(v)}
-
\frac{J(f,f)(v)}{f(v)}
\right|
\\
&\leq
\frac{
|J_M(f_M^N,f_M^N)(v)-J(f,f)(v)|
}{
|f_M^N(v)|
}
+
|J(f,f)(v)|
\left|
\frac{1}{f_M^N(v)}-\frac{1}{f(v)}
\right|
\\
&\leq
\frac{2}{\kappa}
|J_M(f_M^N,f_M^N)(v)-J(f,f)(v)|
+
\frac{2}{\kappa^2}
|J(f,f)(v)|
|f_M^N(v)-f(v)|.
\end{align*}
In the last step, we  used that for all $v\in\Omega_\kappa$,
\[
\left|
\frac{1}{f_M^N(v)}-\frac{1}{f(v)}
\right|
=
\frac{|f_M^N(v)-f(v)|}
{|f_M^N(v)||f(v)|}
\leq
\frac{2}{\kappa^2}
|f_M^N(v)-f(v)|\, .
\]
By taking the supremum, we obtain
\begin{multline}
\label{eq:Uest1}
\|U_M[f_M^N]-U[f]\|_{C^0(\Omega_\kappa)}
\leq
\frac{2}{\kappa}
\|J_M(f_M^N,f_M^N)-J(f,f)\|_{C^0(\Omega_\kappa)}
\\
+
\frac{2}{\kappa^2}
\|J(f,f)\|_{C^0(\Omega_\kappa)}
\|f_M^N-f\|_{C^0(\Omega_\kappa)}.
\end{multline}

To isolate the error given by the spectral approximation and the
density reconstruction, we work with the periodic flux. Recall
$f_M^N=\mathcal{P}_M f_M^N$ and 
$
J_M(f_M^N,f_M^N)
=
\mathcal{P}_M J_{\mathrm{per}}(f_M^N,f_M^N).
$
Since $f$ is supported in $B(0,R)$ and $L>2R$,
Lemma~\ref{l:truncation} gives
\[
J(f,f)
=
J_{\mathrm{per}}(f_{\mathrm{per}},f_{\mathrm{per}})
\qquad\text{on} \quad [-L,L]^d.
\]
We can therefore split the flux error as
\begin{multline}
\label{eq:Jest}
\|J_M(f_M^N,f_M^N)-J(f,f)\|_{C^0(\Omega_\kappa)}
\leq
\|
(\mathcal{P}_M-I)
J_{\mathrm{per}}(f_M^N,f_M^N)
\|_{C^0(\mathbb T_L^d)}
\\
+
\|
J_{\mathrm{per}}(f_M^N,f_M^N)
-
J_{\mathrm{per}}(f_{\mathrm{per}},f_{\mathrm{per}})
\|_{C^0(\mathbb T_L^d)}.
\end{multline}

For the first term, we apply the projection estimate
\eqref{eq:err_C0} with exponent $s-1>d/2$ and
Lemma~\ref{l:flux_stability} to obtain
\begin{align*}
\|
(\mathcal{P}_M-I)
J_{\mathrm{per}}(f_M^N,f_M^N)
\|_{C^0(\mathbb T_L^d)}
& \leq
C(d,L,s)
M^{-(s-d/2-1)}
\|J_{\mathrm{per}}(f_M^N,f_M^N)\|_{H^{s-1}(\mathbb T_L^d)}
\\
&\leq
C(\gamma,d,L,s)
M^{-(s-d/2-1)}
\|f_M^N\|_{H^s(\mathbb T_L^d)}^2.
\end{align*}
From the assumption
$
\|f_M^N-f\|_{H^s(\mathbb T_L^d)}
\leq 1,
$
we have
$
\|f_M^N\|_{H^s(\mathbb T_L^d)}
\leq
\|f \|_{H^s(\mathbb T_L^d)}+1.
$
Therefore,
\[
\|
(\mathcal{P}_M-I)
J_{\mathrm{per}}(f_M^N,f_M^N)
\|_{C^0(\mathbb T_L^d)}
\leq
C
M^{-(s-d/2-1)},
\]
where
$
C=C\left(\gamma, d,L,s,
\|f \|_{H^s}\right).
$

For the second term, we use the bilinearity of
$J_{\mathrm{per}}$ and apply Lemma~\ref{l:flux_stability} to obtain
\begin{align*}
\|
J_{\mathrm{per}}(f_M^N,f_M^N)
- &
J_{\mathrm{per}}(f_{\mathrm{per}},f_{\mathrm{per}})
\|_{C^0(\mathbb T_L^d)}
\\ & \leq
\|
J_{\mathrm{per}}(f_M^N-f_{\mathrm{per}},f_M^N)
\|_{C^0(\mathbb T_L^d)}
+
\|
J_{\mathrm{per}}(f_{\mathrm{per}},f_M^N-f_{\mathrm{per}})
\|_{C^0(\mathbb T_L^d)}
\\
&\leq
C(\gamma, d,L)
\|f_M^N-f_{\mathrm{per}}\|_{C^1(\mathbb T_L^d)}
\left(
\|f_M^N\|_{C^1(\mathbb T_L^d)}
+
\|f_{\mathrm{per}}\|_{C^1(\mathbb T_L^d)}
\right).
\end{align*}
Using again
$
\|f_M^N-f_{\mathrm{per}}\|_{H^s(\mathbb T_L^d)}
\leq 1
$
and Sobolev embedding $H^s(\Td_L) \hookrightarrow C^1(\Td_L)$ for $s>d/2 +1$, we have
\[
\|f_M^N\|_{C^1(\mathbb T_L^d)}
\leq
C(d,L,s)
\left(
\|f_{\mathrm{per}}\|_{H^s(\mathbb T_L^d)}+1
\right).
\]
Therefore,
\[
\|
J_{\mathrm{per}}(f_M^N,f_M^N)
-
J_{\mathrm{per}}(f_{\mathrm{per}},f_{\mathrm{per}})
\|_{C^0(\mathbb T_L^d)}
\leq
C\,
\|f_M^N-f_{\mathrm{per}}\|_{C^1(\mathbb T_L^d)},
\]
where
$
C=C\left(\gamma, 
d,L,s,\|f\|_{H^s}
\right).
$

Altogether, from \eqref{eq:Jest} we have obtained
\begin{equation}
\label{eq:Jest1}
\|J_M(f_M^N,f_M^N)-J(f,f)\|_{C^0([-L,L]^d)}
\leq
C\left(
M^{-(s-d/2-1)}
+
\|f_M^N-f_{\mathrm{per}}\|_{C^1(\mathbb T_L^d)}
\right).
\end{equation}

We turn now to estimating the term depending on $1/\kappa^2$ on
the right-hand side of \eqref{eq:Uest1}. Similarly, we use again
Lemma \ref{l:flux_stability} to obtain
\[
\|J(f,f)\|_{C^0([-L,L]^d)}
\leq
C(\gamma, d,L)
\|f_{\mathrm{per}}\|_{C^1(\mathbb T_L^d)}^2
\leq
C\left(\gamma,
d,L,s,\|f\|_{H^s(\mathbb T_L^d)}
\right)\,,
\]
leading to 
\begin{equation}
\label{eq:Jest2}
\frac{2}{\kappa^2}
\|J(f,f)\|_{C^0(\Omega_\kappa)}
\|f_M^N-f\|_{C^0(\Omega_\kappa)}
\leq
\frac{C}{\kappa^2}
\|f_M^N-f_{\mathrm{per}}\|_{C^0(\mathbb T_L^d)}.
\end{equation}

By combining \eqref{eq:Jest1} and \eqref{eq:Jest2} in
\eqref{eq:Uest1}, we obtain for some $C = C(\gamma, d,s,R,L,\|f\|_{H^s(\mathbb T_L^d)})$
\begin{equation*}
\|U_M[f_M^N]-U[f]\|_{C^0(\Omega_\kappa)}
\leq C\left(
\frac{1}{\kappa}
M^{-(s-d/2-1)}
+
\frac{1}{\kappa}
\|f_M^N-f_{\mathrm{per}}\|_{C^1(\mathbb T_L^d)}
+
\frac{1}{\kappa^2}
\|f_M^N-f_{\mathrm{per}}\|_{C^0(\mathbb T_L^d)}
\right).
\end{equation*}
We note that  $\|\cdot\|_{C^0} \leq \|\cdot\|_{C^1}$ and  $1/\kappa \leq \|f\|_{C^0(\Rd)}/\kappa^2$ given that $\kappa$ is necessarily bounded above by  $\sup_v |f(v)| =  \|f\|_{C^0(\Rd)}$. Therefore,  the last two terms can both be controlled by 
\[
\frac{1}{\kappa}
\|f_M^N-f_{\mathrm{per}}\|_{C^1(\mathbb T_L^d)}
+
\frac{1}{\kappa^2}
\|f_M^N-f_{\mathrm{per}}\|_{C^0(\mathbb T_L^d)}
\leq \frac{C}{\kappa^2} \|f_M^N-f_{\mathrm{per}}\|_{C^1(\mathbb T_L^d)}
\]
leading to the desired estimate.
\end{proof}

\section{Numerical examples}
\label{sec:num}

Through different numerical experiments we validate the theoretical analysis and compare the accuracy and efficiency of the proposed spectral particle method, Algorithm \ref{alg}, against a baseline algorithm, the blob method type II
 \cite{carrillo2020particle}. Code reproducing the results is available at \url{https://github.com/borghig/fast-spectral-particle-landau}.

We will consider three 2D benchmark tests for the homogeneous Landau equation \eqref{eq:landau} with  Maxwellian interactions ($\gamma = 0$). In Section \ref{sec:bkw} we consider the exact BKW solution, see \cite[Appendix A]{carrillo2020particle}, for which the exact definition of both $f(t,v)$ and the flux $J(f,f)(t,v)$ is available. In Section \ref{sec:trubnikov} we consider the Trubnikov relaxation test \cite{trubnikov1965particle} where, in an initially anisotropic distribution, the energy difference across dimensions should decrease at a prescribed rate. In Section \ref{sec:bump} we perform a bump-on-tail test where 10\% of the initial mass is displaced in a sharp bulk away from the main distribution. Since in the BKW solution the collisional operator is assumed to be multiplied by $C = 1/16$, we use this time-rescaling factor in all tests.

In all the experiments, we will consider $N = N_{\mathrm{side}}^2$ particles initialized on a grid over $[-v_{\max}, v_{\max}]^2$. At $t = 0$, particles are therefore equispaced with $h =  2 v_{\max}/(N_{\mathrm{side}}-1)$, and the initial distribution is set to be
\begin{equation} \label{eq:fNin}
f^N(0,v) = \sum_{i=1}^N w_i \delta(v - v_i(0))\,,\qquad \textup{with} \quad  \sum_{i=1}^N w_i= 1 \;, \quad  w_i \propto h^2 f(0,v_i(0))\,.
\end{equation}
The Fourier computational domain is the torus $\Td_L = [-L,L]^d$ with $L = 10$.  This choice makes aliasing and truncation effects negligible for the considered initial distribution. 
 Non-Uniform Fourier transformations are performed using FINUFFT CPU algorithms from \cite{barnett2019parallel} with default tolerance $\epsilon = 10^{-13}$. 

Time integration is performed with either the explicit Euler method or RK4 with time step $\Delta t  = 10^{-2}$ or $\Delta t  = 10^{-3}$, depending on the number of particles and modes $M$ used. 
The floor parameter in \eqref{eq:UM} is set to $\underline{f} = 10^{-13}$. For the blob method, we will fix $\Delta t  = 10^{-2}$,  Euler integration, and  $\ve = 0.64\cdot h^2$, as suggested in \cite{carrillo2020particle}.

To compute errors and diagnostic quantities we employ a fixed grid made of $120\times 120$ points over $[-L,L]^2$ at which  the reconstructed density $f^N_M$ is evaluated using inverse NUFFT (Type II). This grid is shared across all simulations and baseline algorithms for better comparison.

\subsection{BKW test}
\label{sec:bkw}

The derivation of the 2D BKW solution for the homogeneous Landau equation can be found in \cite[Appendix A]{carrillo2020particle}. The solution is radially symmetric and with a ring-shaped initial distribution $f(0,v) = |v|^2 e^{-|v|^2}/\pi$ with maxima at $|v| = 1$.

\subsubsection*{Spectral accuracy at $t = 0$}

To isolate the error given by the spectral approximation from the KDE one, we test the accuracy of the fast flux computations at time $t = 0$. Indeed, on the initialization grid $f^N(0,v)$ already provides a good approximation of the exact initial data $f(0,v)$, and so does its truncated Fourier density estimation $\P_M f^N$.
Let $v_{\max} = L = 10$ and $N_{\mathrm{side}} = 100$. We consider different numbers of modes per dimension $M = 32, 48, 64, 128$ and compute the truncated Fourier approximations $f^M = \P_M f$, $J_M(f,f)$, and $U_M[f^N]$ given by \eqref{eq:UM}. 

Figure \ref{fig:t0} illustrates the accuracy in terms of relative $L^2$ and $L^\infty$ errors. The polynomial approximation $f^N_M$ of $f$ is spectrally accurate. 
As expected from the theory, in particular Proposition \ref{p:J_error}, the spectral approximation of the flux $J_M(f^N_M,f^N_M)$ is also spectrally accurate.
The reconstructed velocity vector field $U_M[f^N]$ is also spectrally accurate over the sets $\Omega_\kappa = \{v \,|\, f(v) >\kappa>0\}$. The error constant, though, deteriorates as $\kappa\to 0$, in accordance with the statement of Theorem \ref{t:main}.

\begin{figure}
\includegraphics[width = 1\linewidth]{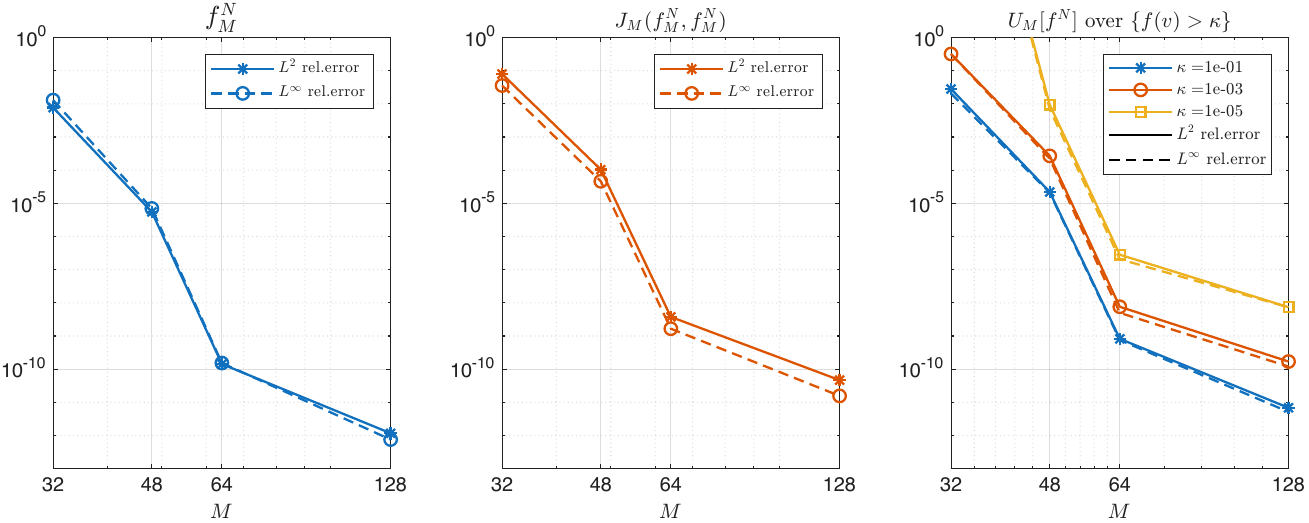}
\caption{
BKW test. At $t  = 0$,  the Fourier KDE $f_M^N$ provides a good approximation of the initial data. This translates into a spectrally accurate approximation $J_M(\P_M f^N,\P_M f^N)$ of the flux $J(f,f)$ at $t = 0$ in accordance with Proposition \ref{p:J_error}. The vector field $U_M[f^N]$ is also spectrally accurate in areas of the domain with strictly positive mass, as in  Theorem \ref{t:main}. Particles are initialized with $N_{\mathrm{side}} = 100$ ($N = 10^4$) and $v_{\max} = L =10$. 
}
\label{fig:t0}
\end{figure}

\begin{figure}
\centering
\begin{subfigure}{\linewidth}
    \centering
    \includegraphics[width=\linewidth]{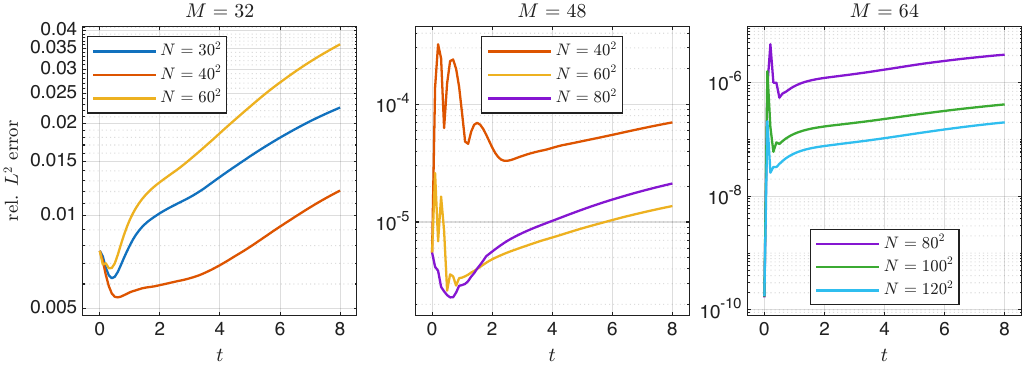}
    \caption{Error for $t \in [0,8]$.}
    \label{fig:comparisonMvsN_time}
\end{subfigure}

\smallskip

\begin{subfigure}{\linewidth}
    \centering
    \includegraphics[width=\linewidth]{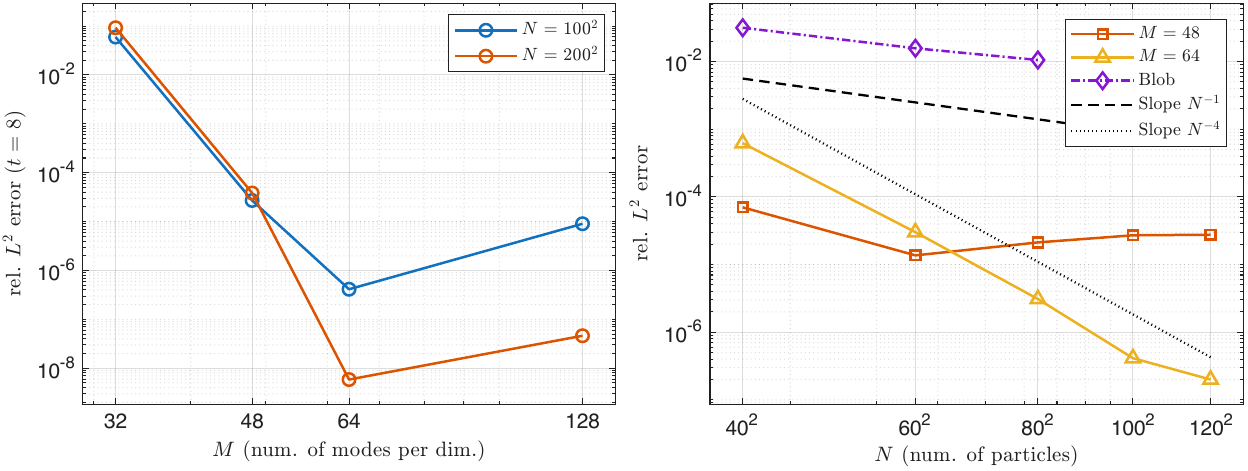}
    \caption{Error at $t = 8$.}
    \label{fig:comparisonMvsN_fixed}
\end{subfigure}

\caption{BKW test. Relative $L^2$ error for different numbers of particles $N$
and modes $M$ per dimension. Simultaneously increasing $N$ and $M$
allows the spectral particle method to achieve high accuracy.
If the number of modes is insufficient, increasing the number of particles may not decrease the error. To isolate the effects of $N$ and $M$, we use $\Delta t = 10^{-3}$ with RK4 time integration for the spectral particle method. For the blob method 
$\Delta t = 10^{-2}$ with Euler time integration suffices to isolate the same effects. The error is computed here on a grid of $200\times 200$ points over $[-10,10]^2$.}
\label{fig:comparisonMvsN}
\end{figure}

\begin{figure}[t]
\includegraphics[width = \linewidth]{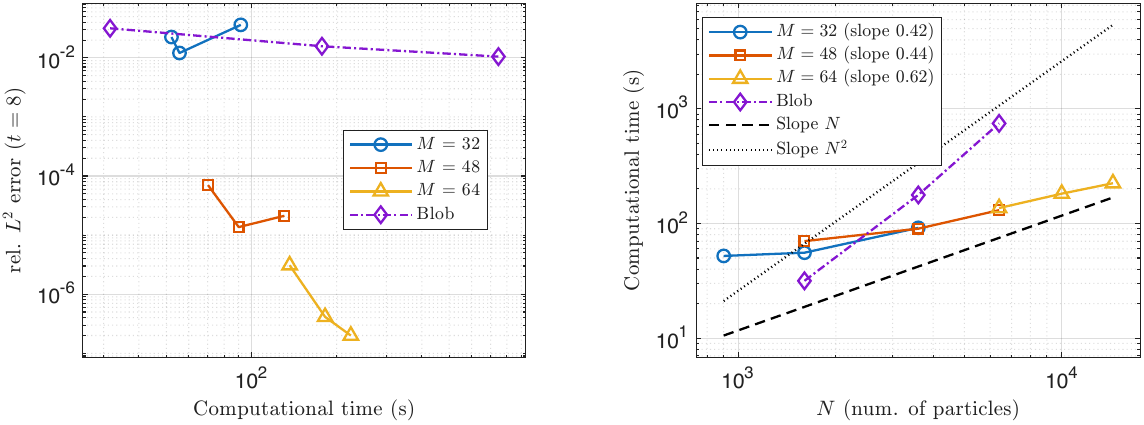}
\caption{BKW test. Accuracy and efficiency of the spectral particle and blob methods for the BKW solution at time $t = 8$ with wall-clock runtime. Different numbers of modes $M = 32, 48, 64$ and particles $N$ are considered (same as Figure \ref{fig:comparisonMvsN_time}). Over the tested range, the spectral particle method shows sub-linear time with respect to $N$ and higher accuracy. From $N> 3\cdot 10^{3}$, the computational time is lower despite using $\Delta t = 10^{-3}$ and RK4 integration (for the blob method, $\Delta t  = 10^{-2}$ with Euler).
 Runtime measurements were performed on a MacBook Air with an Apple M2 processor and 8 GB unified memory, running MATLAB R2026a on macOS 26.6. 
    The FINUFFT CPU package \cite{barnett2019parallel} is used for the NUFFT computations.
}
\label{fig:comparisonMvsN_cost}
\end{figure}

\subsubsection*{Accuracy at $t>0$}

As mentioned, $f^N_M$ might not be spectrally accurate at $t>0$ due to the distortion that the flow induces in the particle grid initialization. 
Figure \ref{fig:comparisonMvsN_time} illustrates the relative $L^2$ error of $f^N_M$ with respect to the exact solution over the time window $[0,8]$. 
To investigate the role of $N$ and $M$ we perform the same experiment with different numbers of particles and modes $M = 32, 48, 64$ and set $\Delta t  = 10^{-3}$ with RK4 time integration. 
Figure \ref{fig:comparisonMvsN_fixed} instead shows the convergence at $t = 8$ as $M$ increases (left plot) and as $N$ increases (right plot). Comparison with the baseline blob method is also included in the latter.

With already $M = 48$ modes per dimension and $N = 60^2 = 3.6\cdot 10^{3}$ particles the method is able to achieve an accuracy of order $10^{-5}$. Interestingly, increasing the number of particles to $N = 80^2 = 6.4\cdot 10^{3}$ does not improve the solution. The same happens with $M = 32$ and $N = 60^2$. When a larger number of modes $M = 64$ is used, instead, the solution improves as $N$ increases, reaching an accuracy of order $10^{-7}$ with $N = 120^2 = 1.44 \cdot 10^4$ particles. Particles are initialized with $v_{\max} = 5$.

\subsubsection*{Computational cost}

In addition to the high-accuracy solutions, the spectral particle method benefits from low computational cost $\mathcal{O}(N + M^d \log M)$ thanks to the implementation via NUFFT algorithms.  Wall-clock time of the computations against the $L^2$ error is shown in Figure \ref{fig:comparisonMvsN_cost}, left panel, while the right panel shows how the computational time increases with respect to the number of particles $N$. The experiments are the same as the ones in Figure \ref{fig:comparisonMvsN} and are compared with the blob method with $N = \{40^2, 60^2, 80^2\}$, $\Delta t = 10^{-2}$, Euler integration. The spectral particle method is able to achieve a high accuracy with a lower computational time,
despite using $\Delta t  = 10^{-3}$ and the RK4 method. Moreover, the
 wall-clock time scales below the linear theoretical bound $\mathcal{O}(N)$ with respect to the number of particles, instead of quadratically as $\mathcal{O}(N^2)$.

\begin{figure}[t]
\includegraphics[width = \linewidth]{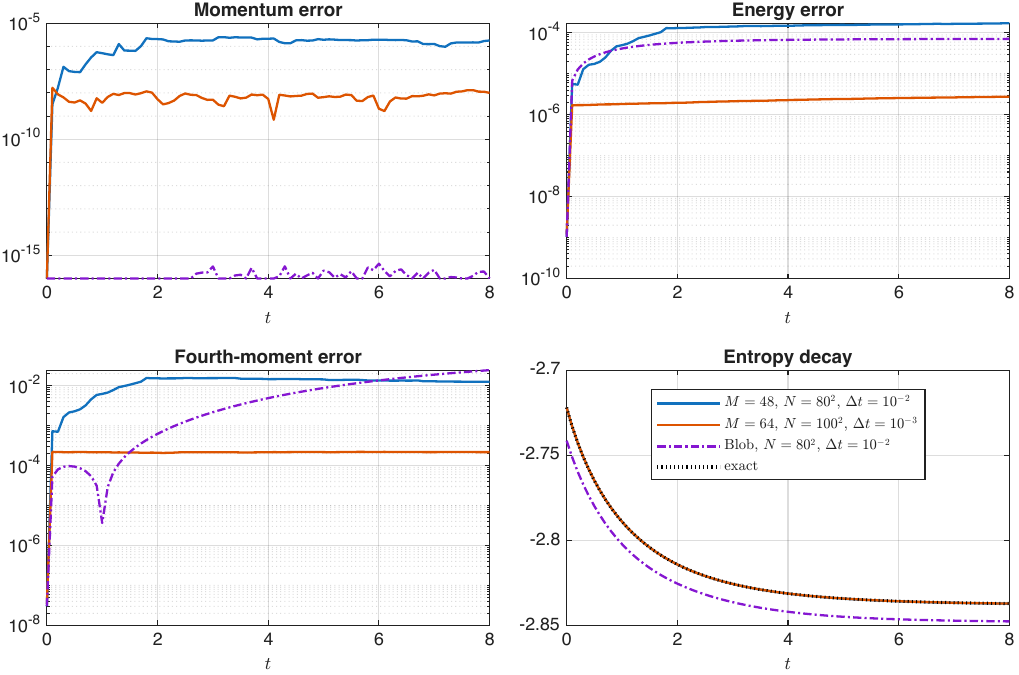}
\caption{
BKW test. Error in tracking macroscopic quantities $\phi = v,|v|^2, |v|^4$ (momentum, energy and fourth moment) and entropy decay against the exact solution. We consider two different parameter settings and include the blob method with $N = 80^2$ as a baseline. 
}

\label{fig:conservation}
\end{figure}

\subsubsection*{Conserved quantities and entropy decay}

Next, we verify the error in the computation of macroscopic quantities. We recall that momentum and energy are conserved by the Landau equation, while the log-entropy decays as the distribution approaches the Maxwellian equilibrium. We also check the error in the tracking of the fourth moment $\int |v|^4 f$. Moments are computed directly via the empirical distribution $f^N$, while the entropy is computed via the Fourier reconstructed density $f^N_M$.
 Results are shown in Figure \ref{fig:conservation} for two different settings: $M = 48, N = 80^2, \Delta t  = 10^{-2}$ and $M = 64, N = 100^2, \Delta t  = 10^{-3}$. For both, RK4 time integration is used. We also include the blob method with $N = 80^2$ for reference.

 The high accuracy of the proposed spectral particle method transfers also into high-accuracy tracking of macroscopic quantities. The blob method outperforms the spectral particle method in tracking the momentum, as its preservation is exact even at the fully discrete particle level. We note that time integration variants conserving energy and entropy decay at the fully discrete level have been proposed in \cite{hu2025fully}.

\subsubsection*{Effect of filtering}

We have noticed how the method's accuracy decreases if the number of particles $N$ is not sufficiently high with respect to the modes $M$. This effect is a consequence of the fact of directly reconstructing the density from the Fourier truncated expansion of $f^N$, and can be addressed by spectral filtering. We perform experiments where we apply a Gaussian filter to the reconstructed density $\hat f ^N_M$, which is given by
\[
(\hat \psi)_k = \exp\left(- \frac{\pi^2 \ve }{2L^2} |k|^2\right)\,,
\]
with different values of the parameter $\ve \in \{ 0.1, 0.3,0.64\}\cdot h^2$.
Figure \ref{fig:filter} shows the evolution of the relative $L^2$ error over the time window $[0,8]$ with $M = 48$  and $M = 128$. The number of particles is set to $N = 80^2 = 6.4\cdot 10^{3} $ in both cases, so that for $M = 128$ there is a relatively low number of particles compared to the number of modes, and the unfiltered reconstruction is not stable. Adding spectral filtering with $\ve > 0.1\cdot h^2$ makes the method stable at the cost of significantly deteriorating its accuracy.
When $\ve = 0.64\cdot h^2$, as in the blob method, the two algorithms return a solution of comparable accuracy. This is to be expected, since applying a  Gaussian filter corresponds to a convolution with a Gaussian kernel in physical space.  We also note how, in Figure \ref{fig:filter-M48}, the high spectral accuracy is lost during the first steps of the computation, but then it stabilizes to an error of order $10^{-4}$. A similar behavior is also observed in Figure \ref{fig:comparisonMvsN_time} where $M = 64$. Here, initialization is done with $v_{\max} = 5$ and Euler integration with $\Delta t  = 10^{-2}$ is used.

\begin{figure}
    \centering
    
    \begin{subfigure}{\linewidth}
        \centering
        \includegraphics[width=\linewidth]{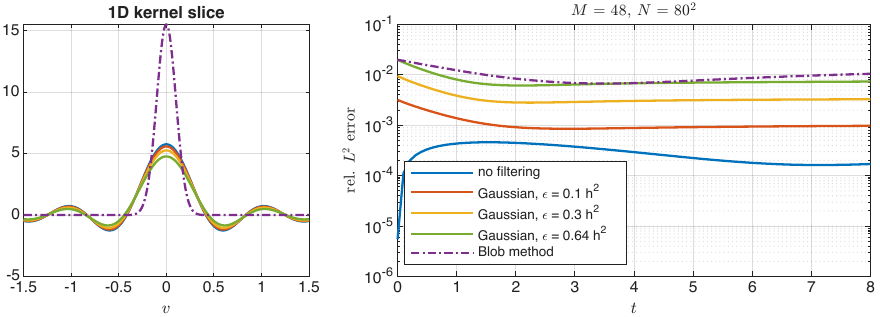}
        \caption{$M=48$}
        \label{fig:filter-M48}
    \end{subfigure}

    \medskip

    \begin{subfigure}{\linewidth}
        \centering
        \includegraphics[width=\linewidth]{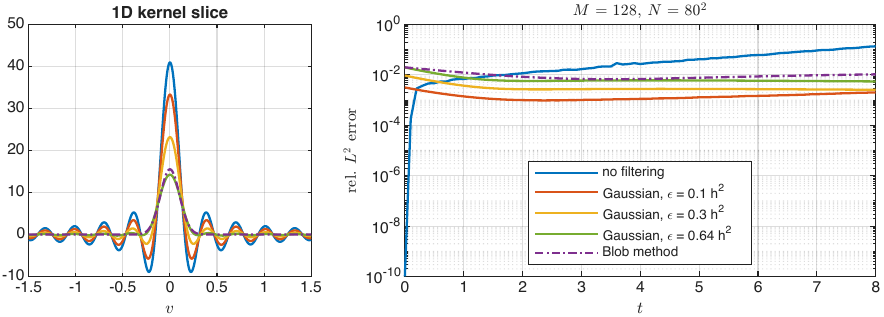}
        \caption{$M=128$}
        \label{fig:filter-M128}
    \end{subfigure}

    \caption{BKW test.
        Effect of spectral filtering with $M = 48, 128$ modes per dimension and $N = 80^2$ particles.  While the initial high spectral accuracy is lost at $t>0$, the unfiltered Fourier density reconstruction $f_M^N$ is stable with $M = 48$, but unstable with $M = 128$. Applying a Gaussian spectral filter mitigates this effect at the expense of accuracy. With high filtering, the error is comparable to the blob method. Here $v_{\max}=5$ and Euler integration with
        $\Delta t=10^{-2}$ is used. We also plot the corresponding 1D slice of the particle kernels $\psi$ for the reconstructed density~\eqref{eq:fNve}. In the case of the unfiltered method, this corresponds to the trigonometric polynomial in \eqref{eq:fNM}.
    }
    \label{fig:filter}
\end{figure}

\subsection{Trubnikov test}
\label{sec:trubnikov}

While in the BKW test the distribution is radially symmetric, in the Trubnikov test \cite{trubnikov1965particle} the initial particle distribution is anisotropic. The experiment verifies whether the method is able to reproduce the exact transfer of energy from one coordinate to the other. We consider the settings of \cite[Test 5]{bailo2025uncertainty} with the only difference of using a time-rescaling constant $C = 1/16$ instead of $C = 1/2$, so that the evolution is analogous to the BKW experiment.
We set the initial temperature in the $v_1$-direction to be $T_1(0)  = 0.75$ and in the $v_2$-direction to be $T_2(0) = 0.5$ so that $\Delta T(0) = T_1(0)  - T_2(0) = 0.25$. Since the equilibrium is given by the Maxwellian distribution which is an isotropic Gaussian, we have $\Delta T(t) \to 0$ as $t \to \infty$ with $T_1(\infty) =T_2(\infty) =  (T_1(0) + T_2(0))/2 = 0.625$. Moreover, for Maxwellian molecules in 2D, with $C = 1/16$ and unit mass, the decay is known exactly 
\begin{equation} \label{eq:trubnikov}
\Delta T(t) = \Delta T(0) e^{-t/2}\,.
\end{equation}
See \cite[Appendix A]{bailo2025uncertainty} for a derivation of Trubnikov’s formula in these settings.

Figure \ref{fig:trubnikov} illustrates the evolution of $\Delta T(t)/\Delta T(0)$ for different combinations of $N$, $M$ parameters over the time window $[0,8]$, and a comparison with the baseline blob method. The spectral particle method is able to simulate the relaxation of the anisotropy with high accuracy, with a final error of order $10^{-8}$ with $N = 100^2 = 10^4$ particles, $M = 64$ modes per dimension, and $\Delta t  = 10^{-3}$.

\begin{figure}
\includegraphics[width = \linewidth]{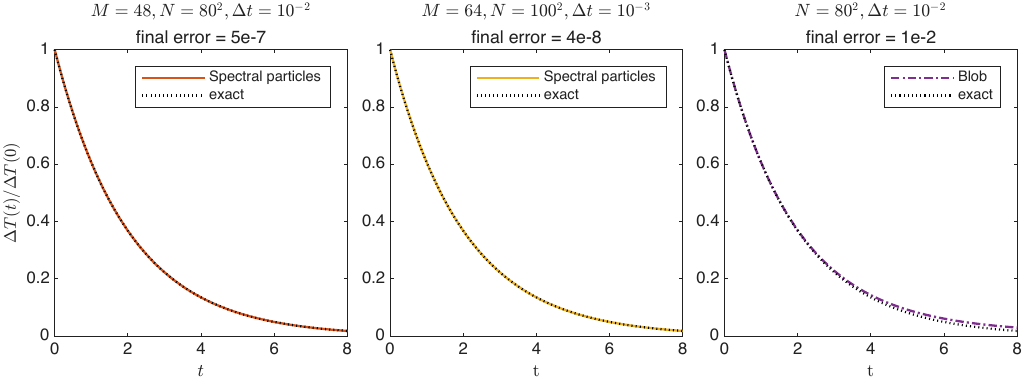}
\caption{Trubnikov test. Relaxation toward equilibrium according to the exact formula. The initial temperature anisotropy $\Delta T(0) = T_1(0) - T_2(0)$ following the exact decay \eqref{eq:trubnikov}. The spectral particle method correctly tracks the decay with increasing accuracy as $N$, $M$ increase. On the right, the blob method is included as a baseline. Particles are initialized on a grid over $[-5,5]^2$ and RK4 time integration is used for the spectral particle method.}
\label{fig:trubnikov}
\end{figure}

\subsection{Bump-on-tail test}
\label{sec:bump}

In this experiment, the starting distribution is given by a Maxwellian state $\mathcal{M}(v;m_c = (0,0),T_c = 1)$ with an additional sharp bump $\mathcal{M}(v; m_b, T_b)$ on its tail which accounts for 10\% of the total mass.
The bump is displaced in the $v_1$-direction with mean velocity $m_b = (3,0)$ and variance $T_b = 0.1$, similar to the settings in \cite{caflisch2008hybrid,MedagliaPareschiZanella2024}. The starting distribution is therefore given by
\begin{equation} \label{eq:bump}
f(0,v) =  0.9 \mathcal{M}(v;(0,0),1)  + 0.1 \mathcal{M}(v;m_b,T_b)\,.
\end{equation}
The total momentum of the system is given by $m = (0.3,0)$ while the energy is given by
$
E = 0.9\cdot (2T_c) + 0.1\cdot (|m_b|^2 + 2 T_b) = 2.72.
$
Therefore, the equilibrium corresponds to the Maxwellian distribution with $m = (0.3,0)$ and $T = (E - |m|^2)/2 = 1.315$. 

Due to the narrow bump, we initialize the particles over a larger domain $[-6,6]^2$ and employ $140$ particles per dimension ($N = 1.96 \cdot 10^4$). Modes are set to $M = 64$. Figure \ref{fig:bump:2d} shows the 2D-distribution at different snapshots $t = 0,1,3,8$, while Figure \ref{fig:bump:1d} shows the $v_1$-marginal distribution against the equilibrium.
The spectral particle method is able to accurately describe the relaxation of the initial distribution
towards the Maxwellian equilibrium, and, in particular, the progressive absorption
of the localized bump into the main bulk.

\begin{figure}
    \centering

    \begin{subfigure}{\linewidth}
        \centering
        \includegraphics[width=\linewidth]{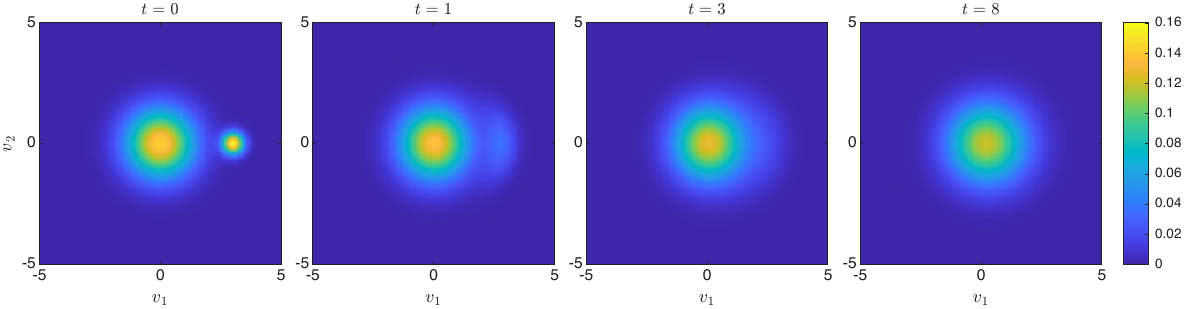}
        \caption{$f(t,v_1,v_2)$}
        \label{fig:bump:2d}
    \end{subfigure}

    \medskip

    \begin{subfigure}{\linewidth}
        \centering
        \includegraphics[width=\linewidth]{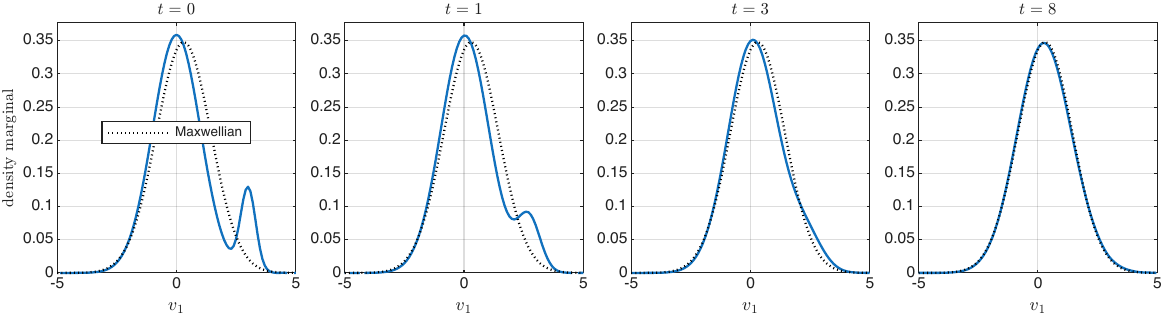}
        \caption{$\int f(t,v_1,v_2)dv_2$}
        \label{fig:bump:1d}
    \end{subfigure}

    \caption{Bump-on-tail test, where the initial distribution \eqref{eq:bump} includes a localized bump of mass on its tail. The method correctly transports and diffuses the bump into the Maxwellian equilibrium. To accurately approximate the initial data, we use $N = 140^2$ particles with $M = 64$ modes per dimension. We set $\Delta t = 10^{-3}$ and RK4 time integration.   
        }
    \label{fig:bump}
\end{figure}

\section{Final remarks and outlook}
We have introduced a deterministic spectral particle method for the
spatially homogeneous Landau equation. Fourier evaluation of the
collision flux, combined with nonuniform fast Fourier transforms
between particles and Fourier modes, yields a cost of
$\mathcal{O}(N+M^d\log M)$ per time step, where $N$ is the number of particles and $M$ is the number of Fourier modes per dimension. This provides an efficient alternative to direct
pairwise evaluation of the particle interactions at an $\mathcal{O}(N^2)$
cost.

For arbitrary interaction potentials with $-d-1\leq \gamma \leq 1$, we derived a consistency estimate that separates the error due to the spectral approximation of the flux from the error introduced by the particle-density reconstruction. The analysis indicates that, for sufficiently smooth reconstructed densities, the Fourier truncation error rapidly becomes negligible compared with the reconstruction error. Numerical experiments on 2D benchmark tests confirm this behavior and show that the method achieves high accuracy, while exhibiting at most linear computational scaling with respect to the number of particles $N$ at fixed spectral resolution.
Total particle mass is preserved by construction, whereas momentum,
energy and entropy are assessed through numerical simulations.

Several aspects remain to be investigated. A complete convergence analysis should include the evolution of the particle trajectories and the effect of the Fourier reconstruction in regions where the density becomes small. Further developments will also address variants of the algorithm which  conserve momentum and energy during the evolution, not only with high accuracy, but exactly. Numerical applications to 3D soft-potential and Coulomb interactions are also left for future work. A particularly relevant direction is the coupling of the proposed spectral particle method with particle-in-cell (PIC) and particle-in-Fourier (PIF) methods for space-inhomogeneous Vlasov--Poisson--Landau and Vlasov--Maxwell--Landau systems \cite{lehe2016pic,ameres2021particle,mitchell2019efficient,sonnen2024variational}. 

\subsection*{Acknowledgments}
The work of Lorenzo Pareschi and Giacomo Borghi was supported by the Royal Society under the Wolfson Fellowship “Uncertainty quantification, data-driven simulations and learning of multiscale complex systems governed by PDEs". Lorenzo Pareschi also acknowledges the partial support of the FIS2023-01334 Advanced Grant “Tackling complexity: advanced numerical approaches for multiscale systems with uncertainties" (ADAMUS). Giacomo Borghi acknowledges the support of the Munich Center for Machine Learning and the ERC
Advanced Grant NEITALG, grant agreement No. 101198055, funded by the European Union.

\bibliographystyle{abbrv}
\bibliography{bibfile}

\end{document}